\documentclass{amsart}

\usepackage{amsmath,amssymb,amsthm}

\usepackage{pinlabel}
\usepackage{enumerate}

\newtheorem{prop}{Proposition}
\newtheorem{thm}[prop]{Theorem}

\newtheorem{add}[prop]{Addendum}

\theoremstyle{definition}

\newtheorem{rem}[prop]{Remark}

\newtheorem*{ack}{Acknowledgements}
\newtheorem*{AI}{AI Declaration}

\newcommand{\co}{\colon\thinspace}

\newcommand{\wtA}{\widetilde{A}}

\newcommand{\cmin}{c_{\min}}

\newcommand{\D}{\mathbb{D}}

\newcommand{\rmd}{\mathrm{d}}

\newcommand{\N}{\mathbb{N}}

\newcommand{\Q}{\mathbb{Q}}

\newcommand{\R}{\mathbb{R}}

\newcommand{\SL}{\mathrm{SL}}

\newcommand{\Tmin}{T_{\mathrm{min}}}

\newcommand{\Z}{\mathbb{Z}}

\DeclareMathOperator{\area}{area}
\DeclareMathOperator{\vol}{vol}

\begin{document}

\author[H.~Geiges]{Hansj\"org Geiges}
\address{Mathematisches Institut, Universit\"at zu K\"oln,
Weyertal 86--90, 50931 K\"oln, Germany}
\email{geiges@math.uni-koeln.de}
\author[J.~Hedicke]{Jakob Hedicke}
\address{Institute for Advanced Study, 1 Einstein Drive,
Princeton, NJ 08540, USA}
\email{jakob.hedicke@gmail.com}
\author[M.~Sa\u{g}lam]{Murat Sa\u{g}lam}
\address{Mathematisches Institut, Universit\"at zu K\"oln,
Weyertal 86--90, 50931 K\"oln, Germany}
\email{msaglam@math.uni-koeln.de}

\title[Integrable contact forms and systolic ratio]{Bott-integrable
contact forms with large systolic ratio}

\date{}

\begin{abstract}
We show that there is no universal upper bound for the systolic
ratio of Bott-integrable contact forms on closed $3$-manifolds,
thus providing further evidence for the relative flexibility
of integrable contact forms. For the proof, we study piecewise
linear approximations of Lutz forms and establish integrability
of a `plug' constructed by Abbondandolo, Bramham, Hryniewicz
and Salom\~ao for pushing up the systolic ratio.
\end{abstract}



\maketitle


\section{Introduction}
The \emph{systole} of a closed Riemannian or Finsler manifold is the
length of the shortest non-contractible
closed geodesic. The \emph{systolic ratio}
is the scale-invariant dimensionless quotient
between the $n^{\mathrm{th}}$ power of the systole and the volume of the
$n$-dimensional manifold. The systolic ratio has a long history,
especially in the Riemannian geometry of surfaces, and under
a variety of assumptions upper bounds are known; see \cite{abhs19}
for an overview, also with a view towards Finsler geometry.

The (co-)geodesic flow can be interpreted as the Reeb flow of the
Liouville contact form on the unit cotangent bundle;
see \cite{geig-rsc} for the Riemannian and \cite{dgz17} for the
Finsler case. Based on this fact, in \cite{apba14} it was first
observed that contact geometry is a natural setting to study
questions in systolic geometry; see also \cite{abhs18,abhs19}.

Given a closed manifold $M$ of dimension $2n+1$ admitting a contact
form~$\alpha$, we denote by $R=R_{\alpha}$ the Reeb vector field of~$\alpha$,
and by $\vol(M,\alpha)$ the volume of $M$ with respect to the
volume form $\alpha\wedge(\rmd\alpha)^n$. If there are
any periodic Reeb orbits---which is always the case in dimension
three~\cite{taub07} and conjecturally in all (odd) dimensions---the
infimum over the periods of all closed Reeb orbits is finite,
and actually a minimum by an argument as in \cite[p.~109]{hoze94}.
We denote this (positive) minimum by $\Tmin(M,\alpha)$. The scale-invariant
systolic ratio is then defined as
\[ \rho(M,\alpha):=\frac{\Tmin(M,\alpha)^{n+1}}{\vol(M,\alpha)}.\]

In general, there is no upper bound on this systolic ratio, that is,
given a contact manifold $(M,\xi)$, one can find contact forms
$\alpha$, defining $\xi=\ker\alpha$, with arbitrarily large
systolic ratio; see~\cite{abhs18,abhs19,sagl21,sagl24}.
However, for many classes of Reeb flows in dimension three---to
which we now restrict our attention---the imposition
of additional symmetries leads to such bounds. Examples
are Reeb flows of contact forms that are
$C^2$-close to a Zoll contact form~\cite{abbe23,abe25},
or Reeb flows of $S^1$-invariant contact forms
on Seifert bundles with non-vanishing Euler number~\cite{vial25,vial-tams}.

A symmetry assumption weaker than $S^1$-invariance, but still believed to
be quite strong, is Bott integrability of the contact form
as defined in~\cite{ghs24}.
The main result of the present paper says that there is no universal
systolic inequality on the class of Bott-integrable contact forms
on any given closed contact $3$-manifold. This shows that integrable
contact forms abound within the class of contact forms defining
a given integrable contact structure, and it provides further evidence that
the integrability requirement may be less restrictive than
expected. This theorem complements
the results about overtwisted Bott-integrable contact
structures in \cite{ghs26}, which can be read as
statements about the abundance of integrable contact structures
within the space of all contact structures.

Recall from \cite{ghs24,ghs26} that a contact form $\alpha$ on a
compact $3$-manifold $M$ is called \emph{Bott integrable}
if there is a Morse--Bott function $f\co M\rightarrow\R$ invariant
under the Reeb flow of~$\alpha$. A contact structure $\xi$ on $M$
said to be \emph{Bott integrable} if there is some Bott-integrable
contact form $\alpha$ with $\xi=\ker\alpha$.

\begin{thm}
\label{thm:main}
Let $\xi$ be a Bott-integrable contact structure on
a closed $3$-mani\-fold~$M$. Then for every $C\in\R^+$ there
exists a Bott-integrable contact form $\alpha$ defining
$\xi=\ker\alpha$ with systolic ratio $\rho(M,\alpha)>C$.
\end{thm}

The strategy for the proof of this theorem is as
follows. As shown in \cite{ghs24}, Bott-integrable contact forms have a
simple description as so-called Lutz forms away from the singular level
sets of the Morse--Bott function. In Sections \ref{section:systolic-Lutz}
and \ref{section:linear-Lutz} we show that outside a set of
small volume these Lutz forms can be replaced by \emph{linear}
Lutz forms, while keeping control over~$\Tmin$. Into these
linear pieces we then insert plugs as constructed by
Abbondandolo, Bramham, Hryniewicz and Salom\~ao in
\cite{abhs18,abhs19}. These plugs are solid tori with a contact
form of small volume, standard near the boundary, and with
$\Tmin\geq 1$. By filling all but a small volume part of
the linear Lutz pieces with such plugs, the total volume can be made
arbitrarily small, whereas $\Tmin$ stays bounded from below.
In order to apply this construction in the Bott-integrable setting,
we need to construct an invariant Morse--Bott function
on the plug; this will be done in Section~\ref{section:plug-integrable}.
The estimates on volume and $\Tmin$ that establish Theorem~\ref{thm:main}
are given in Section~\ref{section:proof}.

A long-standing open question of Katok \cite{kato09} asks whether every
low-dimensional volume-preserving dynamical system with zero entropy
can be approximated by integrable systems.
A version of that question on surfaces was studied in
\cite{brkh23} and \cite{khan22}, where the authors show that there is a
gap between autonomous Hamiltonian diffeomorphisms and integrable
Hamiltonian diffeomorphisms in the $C^0$-topology
and the Hofer metric, respectively.
Theorem \ref{thm:main}, taken together with the results of
Abbondandolo, Benedetti and Edtmair \cite{abbe23,abe25} on Zoll
contact forms, implies
that there exist integrable contact forms
that are far from Zoll contact forms in the $C^2$-topology.
It would be interesting to understand whether there are
integrable contact forms far from $S^1$-invariant ones.
The same question can be asked with the $C^2$-topology replaced by
the Banach--Mazur pseudometric on contact forms
introduced in~\cite{abe25}, which may be regarded as an
analogue of the Hofer metric on Hamiltonian diffeomorphisms.
\section{The systolic geometry of Lutz forms}
\label{section:systolic-Lutz}
A \emph{Lutz form} is a contact form $\alpha$ on
$W:=[-1,1]\times T^2$ that can be written as
\[ \alpha=h_1(r)\,\rmd x_1+h_2(r)\,\rmd x_2,\]
where $r$ is the coordinate on $[-1,1]$, and $x_1,x_2$ are those on
$T^2=(\R/\Z)^2$. With the orientation of $W$ defined by the
volume form $\rmd r\wedge\rmd x_1\wedge\rmd x_2$, the contact condition
becomes
\[ \Delta_h:=\begin{vmatrix}h_1&h_1'\\ h_2&h_2'\end{vmatrix}<0,\]
since $\rmd\alpha=h_1'\,\rmd r\wedge\rmd x_1+h_2'\,\rmd r\wedge\rmd x_2$
and $\alpha\wedge\rmd\alpha=-\Delta_h\,\rmd r\wedge\rmd x_1\wedge\rmd x_2$.

By the Reeb--Liouville theorem \cite[Theorem~2.2]{ghs24},
every connected component of a regular level set of a Bott integral
is a torus, the Bott-integrable contact form is a Lutz form in a
neighbourhood of this torus, and the Bott integral is a strictly
monotone function of~$r$.

The Reeb vector field of a Lutz form is given by
\begin{equation}
\label{eqn:R-Lutz}
R=\frac{h_2'}{\Delta_h}\,\partial_{x_1}-\frac{h_1'}{\Delta_h}\,\partial_{x_2}.
\end{equation}
Moreover, we have
\[ \vol(W,\alpha)=-\int_{-1}^1\Delta_h\,\rmd r.\]
From these formulae it follows immediately that both $R$
(and hence $\Tmin$) and $\vol(W,\alpha)$ are invariant under
reparametrisations $[a,b]\ni s\mapsto r(s)\in[-1,1]$. Indeed, with
$k_i(s):=h_i(r(s))$, $i=1,2$, the map $s\mapsto r(s)$ induces
a strict contactomorphism between the Lutz forms
$k_1(s)\,\rmd x_1+k_2(s)\,\rmd x_2$ and $h_1(r)\,\rmd x_1+h_2(r)\,\rmd x_2$
on $[a,b]\times T^2$ and $[-1,1]\times T^2$, respectively.

In the case that $h_2'(r)/h_1'(r)\in\Q\cup\{\infty\}$, we
write $\Tmin(r)$ for the minimal period of $R$ on $\{r\}\times T^2$.
When this ratio is irrational, we set $\Tmin(r)=\infty$.
In the rational case, we have
\begin{equation}
\label{eqn:Tmin}
\begin{array}{rcl}
\Tmin(r) & =    & \min\bigl\{t>0\co th_i'(r)/\Delta_h\in\Z,
                  \, i=1,2\bigr\}\\
         & \geq & \min\bigl\{|\Delta_h(r)/h_1'(r)|,
                  |\Delta_h(r)/h_2'(r)|\bigr\}.
\end{array}
\end{equation}
When both $h_1'(r)$ and $h_2'(r)$ are non-zero, a better estimate is
given by $\max$ in the last line, but the estimate with $\min$
also holds for one of the two being zero. For
$h_1'(r)=0$, say, we have $R=(1/h_1(r))\,\partial_{x_1}$, and the
minimal period is $\Tmin(r)=|h_1(r)|$.
\section{Linearising the contact form}
\label{section:linear-Lutz}
Let $\alpha$ be a Bott-integrable contact form on a closed
$3$-manifold~$M$, with Bott integral $f\co M\rightarrow\R$.
The aim of this section is to show that one can homotope $\alpha$
to a \emph{linear} Lutz form
\[ a\,\rmd x_1-(br+c)\,\rmd x_2,\]
$a,b\in\R^+$, $c\in\R$, with Reeb vector field $a^{-1}\partial_{x_1}$,
on a collection of product neighbourhoods $[-1,1]\times T^2
\subset M$ whose complement is
a set of arbitrarily small volume with respect to
the new contact form.

For the time being, we assume that $(M,\alpha)$ does
\emph{not} contain any linear Lutz pieces of \emph{irrational}
slope (ILL\footnote{Irrational Linear Lutz piece or, according to
the OED, ``ill, n., Middle English: the opposite of good, trouble''.}),
that is, regions $[-1,1]\times T^2\subset M$ where
$\alpha$ is of the form $h_1(r)\,\rmd x_1+h_2(r)\,\rmd x_2$
with the function $h_2'/h_1'$ assuming a constant irrational
value.

\begin{prop}
\label{prop:beta-epsilon}
Given $(M,\alpha,f)$ without ILLs, for every $\varepsilon>0$ one can find a
Bott-integrable contact form $\beta=\beta_{\varepsilon}$ with
the following properties:
\begin{enumerate}[(i)]
\item $\beta$ is homotopic to $\alpha$ via Bott-integrable contact
forms, all sharing the Bott integral~$f$.
\item $\Tmin(M,\beta)>\Tmin(M,\alpha)-\varepsilon$
and $\vol(M,\beta)<\vol(M,\alpha)+\varepsilon$.
\item There is an open subset $U\subset M$ with
$\vol(U,\beta)<\varepsilon$ such that $M\setminus U$ is
the disjoint union of finitely many subsets $W_{\ell}\subset M$
diffeomorphic to $[-1,1]\times T^2$, with
\[ \beta|_{W_{\ell}}=a_{\ell}\,\rmd x_1-(b_{\ell}r+c_{\ell})\,\rmd x_2\]
and $f|_{W_{\ell}}=f(r)$ with $f'(r)\neq 0$.
Here $a_{\ell}\in\R^+$ is the minimal period
of the Reeb flow on~$W_{\ell}$, and $b_{\ell}\in\R^+$ is determined by
the identity $\vol(W_{\ell},\beta)=2a_{\ell}b_{\ell}$.
\end{enumerate}
\end{prop}

\begin{proof}
The Morse--Bott function $f$ on the closed manifold $M$
has finitely many singular levels $c_1,\ldots,c_k\in\R$. For
$\delta>0$ sufficiently small, the union
\[ U_0:=\bigcup_{i=1}^k\bigl\{p\in M\co f(p)\in(c_i-\delta,
c_i+\delta)\bigr\}\]
of neighbourhoods of the singular level sets $f^{-1}(c_i)$ will
satisfy $\vol(U_0,\alpha)<\varepsilon/4$.

By the Reeb--Liouville theorem \cite[Theorem 2.2]{ghs24},
the complement $M\setminus U_0$ is foliated by tori, each of
which has a neighbourhood diffeomorphic to $[-1,1]\times T^2$,
where $f=f(r)$ with $f'(r)\neq 0$, and
$\alpha=h_1(r)\,\rmd x_1+h_2(r)\,\rmd x_2$.
Thanks to $M\setminus U_0$ being compact, we can choose finitely
many of these neighbourhoods whose interiors cover the complement of~$U_0$.
Moreover, it may be assumed that there are no triple intersections
between these Lutz pieces or with~$U_0$, and that the total volume of the
pairwise intersections is smaller than~$\varepsilon/4$.

Next we modify the contact form $\alpha$ on each of the finitely
many product neighbourhoods $[-1,1]\times T^2$ as follows.
Here we use the absence of ILLs.
In a small neighbourhood of $r=\pm 1$, containing the
intersection with $U_0$ or with another Lutz piece, we leave the plane curve
$r\mapsto\bigl(h_1(r),h_2(r)\bigr)$ unchanged. Outside these neighbourhoods,
which we may assume to have total volume smaller than~$\varepsilon/2$,
we choose a $C^1$-close approximation by a piecewise linear
curve $r\mapsto\bigl(g_1(r),g_2(r)\bigr)$ with vertices on the curve
$(h_1,h_2)$ and rational (or infinite) slope $g_2'/g_1'$ of the linear
pieces; the linear pieces need not be
parametrised linearly. The convex linear interpolation between
$(h_1,h_2)$ and $(g_1,g_2)$ defines a homotopy via contact forms
sharing the same Bott integral.

Finally, we smoothen the curve $(g_1,g_2)$ in small neighbourhoods of
the vertices of total volume (with respect to the
volume form $\beta\wedge\rmd\beta$ defined by the contact form
$\beta$ corresponding to the new curves $(g_1,g_2)$ on all
the Lutz pieces) smaller than~$\varepsilon/4$.
The last comment about convex linear interpolation still applies, so
condition (i) in the proposition is satisfied by the resulting
contact form~$\beta$.

With $\{W_{\ell}\}$ we denote
the collection of linear pieces, and with $U$ their complement, made up
of $U_0$, the smoothing regions, and the regions near the
boundary of the original Lutz pieces where we have left
the contact form $\alpha$ unchanged. By construction, we have
$\vol(U,\beta)<\varepsilon$.

For a sufficiently $C^1$-close approximation of the
curves $(g_1,g_2)$ describing the Lutz pieces,
the volume conditions in (ii) and (iii) will be satisfied by
the formula for $\vol(W,\alpha)$ from Section~\ref{section:systolic-Lutz}.

The finitely many Lutz pieces making up the complement of $U$
are now of the form
\[ [r_0,r_1]\times T^2,\;\;\;\beta=(b_1r+c_1)\,\rmd x_1+(b_2r+c_2)\,\rmd x_2,\]
with $b_2/b_1\in\Q\cup\{\infty\}$. By a linear
change of coordinates in $T^2$, using an element of $\SL(2,\Z)$,
and a reparametrisation of the $r$-coordinate, we can bring the Lutz
pieces into the form described in~(iii). The characterisations of
$a_{\ell},b_{\ell}$ in terms of minimal period and volume are
clear from the formulae in Section~\ref{section:systolic-Lutz}.

It remains to show that $\Tmin(M,\beta)$ is bounded from below by
$\Tmin(M,\alpha)-\varepsilon$ for a sufficiently $C^1$-close
approximation~$\beta$.
To this end, we need to relate the periodic orbits of $R_{\alpha}$ and
$R_{\beta}$ on a Lutz piece $[-1,1]\times T^2$, where
$\alpha=h_1(r)\,\rmd x_1+h_2(r)\,\rmd x_2$ and
$\beta=g_1(r)\,\rmd x_1+g_2(r)\,\rmd x_2$. By formula
\eqref{eqn:R-Lutz} from Section~\ref{section:systolic-Lutz} for the
Reeb vector field on a Lutz piece, the flow of $R_{\beta}$ is
periodic on $\{r\}\times T^2$ precisely when the slope
$g_2'(r)/g_1'(r)$ is rational or infinite. In particular,
this is the case on the linear pieces of~$(g_1,g_2)$.

By the choice of vertices of the piecewise linear approximation
on the curve $(h_1,h_2)$, and by the mean value theorem,
the slope $s_{\ell}=g_2'/g_1'$ of a linear piece
$[r_{\ell},r_{\ell+1}]$---before the
smoothing of~$(g_1,g_2)$---is realised by $h_2'(r_{\ell}^*)/h_1'(r_{\ell}^*)$
at some point $r_{\ell}^*\in(r_{\ell},r_{\ell+1})$;
see Figure~\ref{figure:PL-Lutz}. Then
\[ R_{\beta}=\frac{\Delta_h(r_{\ell}^*)}{\Delta_g(r_{\ell}^*)}\cdot
\frac{g_1'(r_{\ell}^*)}{h_1'(r_{\ell}^*)}\cdot R_{\alpha}\]
on $\{r_{\ell}^*\}\times T^2$ for $h_1'(r_{\ell}^*),g_1'(r_{\ell}^*)\neq 0$,
else we replace the second factor by the quotient
$g_2'(r_{\ell}^*)/h_2'(r_{\ell}^*)$.
The approximation being $C^1$-close implies that
$R_{\beta}$ and $R_{\alpha}$ (and hence the corresponding
minimal periods) differ by a factor close to~$1$.

In the smoothing regions of $(g_1,g_2)$, the slope $g_2'/g_1'$
varies between two rational slopes $s_{\ell}$ and $s_{\ell+1}$
of linear segments. By the intermediate value theorem, any rational
slope $s$ of $R_{\beta}$ between $s_{\ell}$ and $s_{\ell+1}$ in this region
equals $h_2'(r_s^{\ell})/h_1'(r_s^{\ell})$ for some
$r_s^{\ell}\in(r_{\ell}^*,r_{\ell+1}^*)$;
again, see Figure~\ref{figure:PL-Lutz}.
As before, the Reeb flows of $\beta$ and $\alpha$ differ by
a reparametrisation factor close to~$1$, provided we choose
the linear segments in $(g_1,g_2)$ sufficiently short. This proves the
estimate for $\Tmin$ in (ii) and hence the proposition.
\end{proof}

\begin{figure}[h]
\labellist
\small\hair 2pt
\pinlabel $(h_1,h_2)$ [t] at 9 335
\pinlabel $(g_1,g_2)$ [tr] at 171 286
\pinlabel $\text{slope $s_{\ell}$}$ [l] at 327 268
\pinlabel $\text{slope $s_{\ell+1}$}$ [r] at 278 196
\pinlabel $\text{slope $s$}$ [l] at 376 157
\endlabellist
\centering
\includegraphics[scale=0.5]{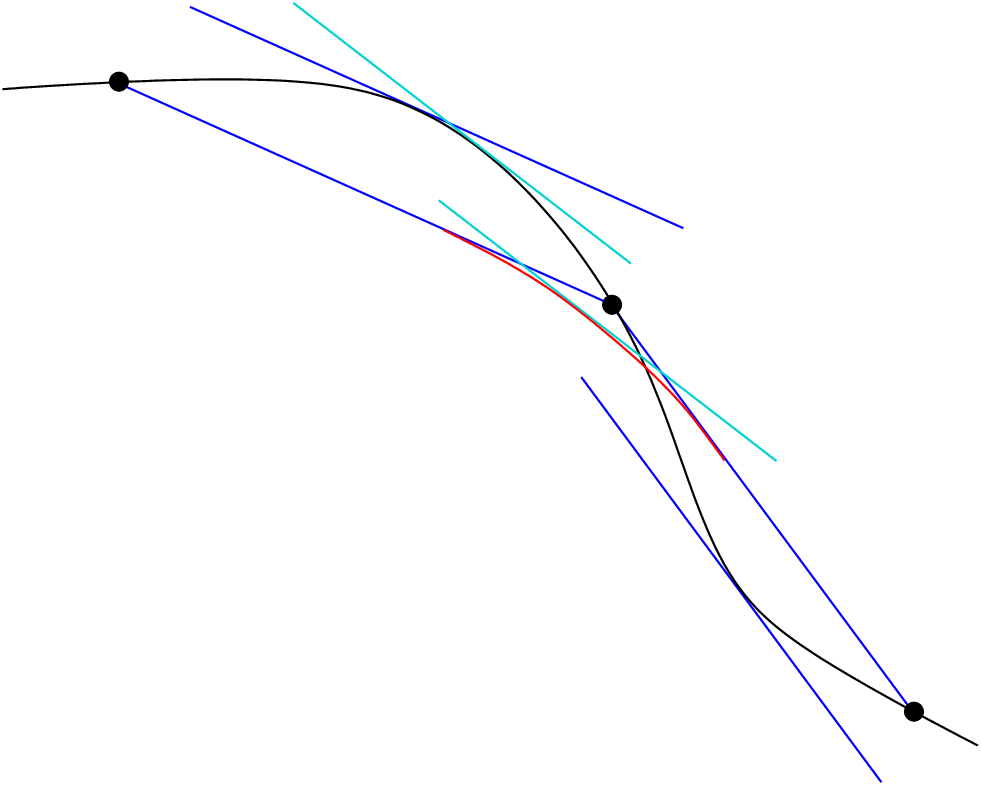}
  \caption{The smoothened piecewise linear Lutz form.}
  \label{figure:PL-Lutz}
\end{figure}
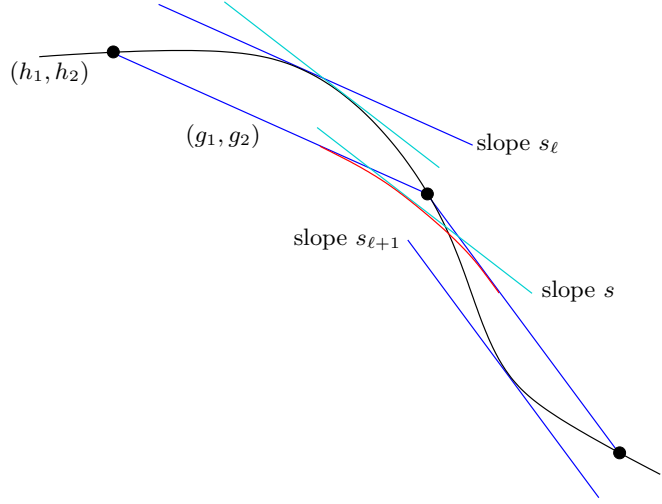

We now address the issue when $(M,\alpha)$ \emph{does} contain ILLs.

\begin{rem}
The argument for proving the Reeb--Liouville theorem \cite[Thm.~2.2]{ghs24},
shows that any regular torus $T^2_0$ is contained in a maximal Lutz piece
$(a,b)\times T^2$, given by following the flow
$\phi_t$ of the normalised gradient $\nabla f/|\nabla f|^2$.
This flow sends regular levels to regular levels, and the flow time gives the
difference in levels.
We identify $T^2_0$ with $\{0\}\times T^2$, and as the interval
parameter we take the flow time, so that $\phi_t(T^2_0)=\{t\}\times T^2$.
The values $a<0<b$ then correspond to the
moment a singular level is reached asymptotically in
backward or forward time, respectively.
\end{rem}

Likewise, any ILL is contained in a maximal one. A maximal ILL may be
contained in a larger Lutz piece, in which case it will be compact,
or it may be asymptotic at one or the other end to a singular level,
in which case the corresponding end will be open. If we take all ILLs
to be maximal, by compactness of $M$ there will be finitely may of them.

\begin{rem}
The case of an ILL being asymptotic to an elliptic orbit can
actually be avoided by perturbing $\alpha$ as in the neighbourhood theorem
for elliptic orbits \cite[Thm.~2.4]{ghs24}.
\end{rem}

We now use the estimates from Section~\ref{section:systolic-Lutz}
to show how the ILLs can be removed by a perturbation
of the contact form, while keeping control over $\Tmin$ and the
contact volume.

We begin with a compact ILL $[-1,1]\times T^2$ contained in a larger
Lutz piece $[-1-\mu,1+\mu]\times T^2$, with $\alpha=h_1(r)\,\rmd x_1+
h_2(r)\,\rmd x_2$ having the property that the function
$h_2'/h_1'$ is identically equal to an irrational constant on $[-1,1]$.
The two functions $r\mapsto|\Delta_h(r)/h_i'(r)|$ take on a positive
minimum on $[-1-\mu,1+\mu]$; write $c$ for the smaller of the two.
We can perturb the~$h_i$, relative to a neighbourhood of
the boundary of the interval $[-1-\mu,1+\mu]$ such that
$h_2'/h_1'$ no longer equals an irrational constant on any
non-trivial interval. For a sufficiently $C^1$-close approximation,
the new Lutz form on $[-1-\mu,1+\mu]\times T^2$ (which has the
same Bott integral $f$ as before) will satisfy $\Tmin\geq c/2$
by estimate~\eqref{eqn:Tmin}.

As long as there are no non-compact ILLs, this is all we have to do,
and we now simply proceed with Proposition~\ref{prop:beta-epsilon}
applied to the perturbed integrable contact form.
For non-compact ILLs, the matter is a little more subtle, because
the neighbourhood $U_0$ we choose in the proof of
Proposition~\ref{prop:beta-epsilon} depends on the given $\varepsilon>0$.

Let us assume we are dealing with a non-compact ILL
$(0,1]\times T^2$, contained in a Lutz piece $(0,1+\mu]\times T^2$
and asymptotic to a singular level at~$0$.
The case where both ends are open and asymptotic to a singular level is
analogous. Given an $\varepsilon$ as in Proposition~\ref{prop:beta-epsilon},
we choose $U_0$ as in the proof there. We then choose a
perturbation of $(h_1,h_2)$, compactly supported  on $(0,1+\mu)$,
such that $h_2'/h_1'$ no longer equals an irrational constant on
any non-trivial interval near the boundary of~$U_0$.
The resulting contact form has the property that $M$ is covered
by $U_0$ and finitely many compact Lutz pieces not containing ILLs.
Notice that $U_0$ may still contain ILLs, but they are not relevant.

This perturbed contact form now depends on $\varepsilon$.
However, no matter
how small $\varepsilon$ (and hence~$U_0$), we can bound
$\Tmin$ of the perturbed form on the Lutz piece in question
from below by the $\varepsilon$-independent $c/2$ as before;
notice that $\Delta_h/h_i'$ is constant when $h_2'/h_1'$ is.

We then take $\cmin$ to be the smallest $c$ of all the finitely many
ILLs we had to perturb away. Proposition~\ref{prop:beta-epsilon}
then still holds true for $(M,\alpha)$ containing ILLs,
provided we replace the bound $\Tmin(M,\alpha)-\varepsilon$ in (ii) by
\[ \min\bigl\{\Tmin(M,\alpha),\cmin/2\bigr\}-\varepsilon.\]
Also, the increase in volume caused by these perturbations can be
controlled independently of~$\varepsilon>0$, so we may assume
that the new volume is not greater than $\vol(M,\alpha)+1$,
and the volume bound in (ii) becomes $\vol(M,\alpha)+1+\varepsilon$.
\section{Integrability of the ABHS plug}
\label{section:plug-integrable}
The following proposition, due to Abbondandolo, Bramham, Hryniewicz and
Salom\~ao (ABHS), lists the essential properties of their plug;
see \cite[Proposition 3.1]{abhs19}. The details of the
construction are contained in \cite[Proposition 2.27]{abhs18}
and \cite[Proposition 3.1]{abhs18}. We write $x,y$ for the Cartesian
coordinates and $r$ for the
radial coordinate on the closed unit disc~$\D$, and $s$ for the
angular coordinate on the circle $\R/\Z$.

\begin{prop}[ABHS]
\label{prop:abhs}
Let $r_0,\delta>0$ and a primitive $\lambda$ of the standard area
form $\rmd x\wedge\rmd y$ on $r_0\D$ be given. Then there is
a contact form $\beta=\beta_{r_0,\lambda,\delta}$ on
the solid torus $r_0\D\times\R/\Z$
with the following properties:
\begin{enumerate}[(i)]
\item $\beta=\lambda+\rmd s$ near  the boundary of the solid torus;
in particular, the Reeb vector
field equals $\partial_s$ near the boundary.
\item $\beta$ is isotopic relative to a neighbourhood of the boundary
and via contact forms to $\lambda+\rmd s$.
\item The minimal period of the Reeb vector field is at least~$1$.
\item The volume $\vol(r_0\D\times\R/\Z,\beta)$ is smaller than~$\delta$.
\end{enumerate}
\end{prop}

We want to prove the following addendum to this proposition.

\begin{add}
\label{add:abhs}
The contact form $\beta$ in the preceding proposition admits a Bott integral
$f$ whose level sets near the boundary of the solid torus
are regular tori $\{r=\mbox{\rm const.}\}$, that is, $f=f(r)$
and $f'(r)>0$ for $r$ close to~$r_0$.
\end{add}

The key to Proposition~\ref{prop:abhs} is the construction
of an area-preserving diffeomorphism $\varphi$ of $r_0\D$, equal to
the identity near
the boundary, with control over the compactly supported
primitive of the closed $1$-form $\varphi^*\lambda-\lambda$,
the periodic points and
the Calabi invariant; see~\cite[Proposition 3.2]{abhs19}.
Choose a \emph{positive}
primitive $\tau\co r_0\D\rightarrow\R^+$ of $\varphi^*\lambda-\lambda$,
constant near the boundary. Then the contact form $\lambda+\rmd s$
on $r_0\D\times\R$ with Reeb vector field $\partial_s$ is
invariant under the diffeomorphism
\[ \begin{array}{rccc}
\Phi\co & r_0\D\times\R & \longrightarrow & r_0\D\times\R\\
        & (z,s)         & \longmapsto     & \bigl(\varphi(z),s-\tau(z)\bigr),
\end{array} \]
and hence descends to a contact form on
$(r_0\D\times\R)/\langle\Phi\rangle$.

In order to prove Addendum~\ref{add:abhs}, we simply need to exhibit
a Morse function $h\co r_0\D\rightarrow\R$ (or the corresponding
level sets) invariant under~$\varphi$, and with regular levels
$\{r=\mbox{\rm const.}\}$ in a neighbourhood of the boundary of the disc.
This Morse function then gives rise to a Bott integral, first on the
mapping torus $(r_0\D\times\R)/\langle\Phi\rangle$,
and then on its diffeomorphic copy $r_0\D\times\R/\Z$,
with the properties claimed in the addendum.

We now describe the relevant properties of~$\varphi$; details
of the construction are given in \cite[pp.\ 734--736]{abhs18}.
For ease of notation, we assume $r_0=1$; the general case follows
by an appropriate rescaling as discussed after \cite[Proposition 3.2]{abhs19}.
The area-preserving diffeomorphism $\varphi$ is a composition
$\varphi=\varphi^+\circ\varphi^-$ of two Hamiltonian diffeomorphisms.

The construction of $\varphi^{\pm}$ depends on a parameter~$n\in\N$;
when chosen sufficiently large, this yields the required estimates in
Proposition~\ref{prop:abhs}. The properties of the diffeomorphism
$\varphi^+$ of $\D$ are:
\begin{enumerate}[(i)]
\item Each circle $\{r=\mbox{\rm const.}\}$ is rotated in \emph{positive}
direction through an angle at most $2\pi/n$.
\item On a sufficiently large disc $\rho\D\subset\D$, the rotation angle
equals $2\pi/n$.
\item Near the boundary of~$\D$, the map $\varphi^+$ is the identity.
\end{enumerate}

Next one chooses a finite collection of pairwise disjoint closed discs
$D_j$, $j\in J$, in the sector
\[ S:=\bigl\{z\in\D\co |z|<\rho,\, \arg(z)\in (0,2\pi/n)\bigr\},\]
with the total area of the $D_j$ close to $\pi\rho^2/n$.
For each $j\in J$ one defines a Hamiltonian diffeomorphism~$\varphi_j$,
compactly supported in the interior of~$D_j$ and acting
on $D_j$ analogously to the action of $\varphi^+$ on~$\D$, but
rotating in \emph{negative} direction. Then
extend this diffeomorphism of $S$ to a diffeomorphism $\varphi^-$
of $\D$ equivariant with respect to the
$\Z_n$-action generated by the rigid rotation of
$\D$ through an angle $2\pi/n$. We may think of the
disc $\D/\Z_n$ as being obtained by gluing the radial boundaries
of the sector $\{\arg(z)\in(0,2\pi/n)\}\subset\D$,
so we regard the $D_j$ as discs in $\D/\Z_n$.

The relevant properties of the Morse function $h\co\D\rightarrow\R$ are
encoded in the Morse foliation of level sets of~$h$, which we now describe.
Away from a small disc around the centre $0\in\D/\Z_n$
disjoint from the $D_j$,
the branched covering $\D\rightarrow\D/\Z_n$ restricts to
a covering $\wtA\rightarrow A_n$ of annuli. Join the `beads' $D_j$ in
$A_n$ into a necklace as shown in Figure~\ref{figure:morse-foliationD}.
This gives rise to a Morse foliation with a critical point of
Morse index~$0$ at the centre of each~$D_j$, and a critical point of
index~$1$ on each connecting piece of the necklace. The corresponding
Morse function on $A_n$ extends to a Morse function
on the disc $\D/\Z_n$, radially symmetric near
the centre, and with an index~$2$ critical point at the centre.
This function lifts to the desired Morse function $h$ on~$\D$.

\begin{figure}[h]
\labellist
\small\hair 2pt
\pinlabel $\theta=0$ [l] at 513 151
\pinlabel $\partial(\D/\Z_n)$ [l] at 508 289
\pinlabel $\theta=2\pi/n$ [r] at 0 151
\endlabellist
\centering
\includegraphics[scale=0.52]{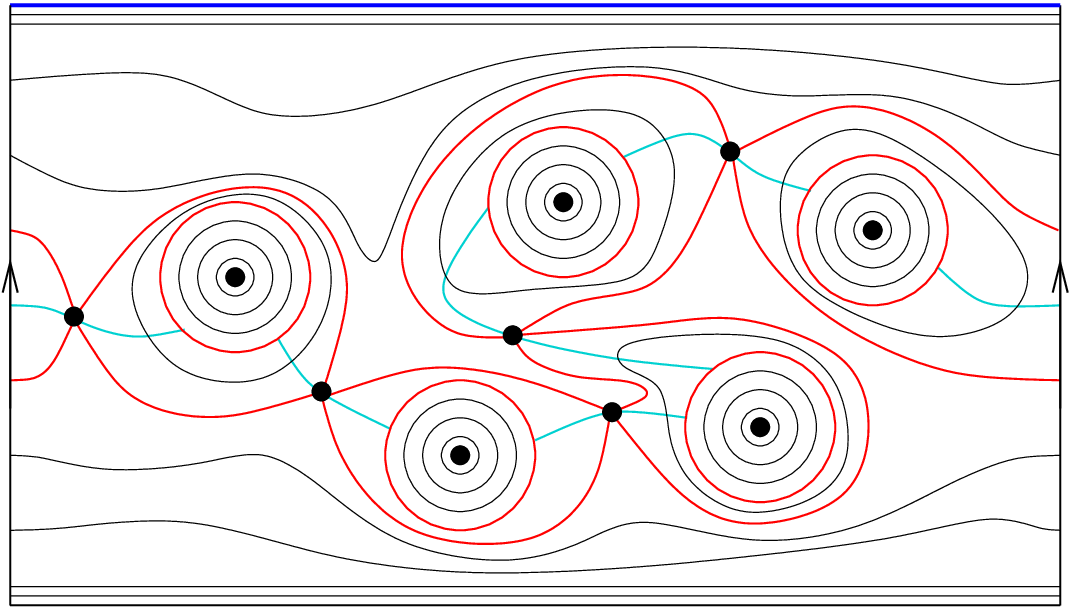}
  \caption{The Morse foliation on the annulus $A_n\subset\D/\Z_n$.}
  \label{figure:morse-foliationD}
\end{figure}
\section{Proof of Theorem~\ref{thm:main}}
\label{section:proof}
Starting from a Bott-integrable contact form $\alpha$ on~$M$,
let $\beta_{\varepsilon}$ be the contact form constructed in
Proposition~\ref{prop:beta-epsilon}.
We consider a linear Lutz piece $W_{\ell}\cong T^2\times[-1,1]$ with
\[ \beta_{\varepsilon}|_{W_{\ell}}=
a_{\ell}\,\rmd x_1-(b_{\ell}r+c_{\ell})\,\rmd x_2,\]
$\Tmin(W_{\ell},\beta_{\varepsilon})=a_{\ell}$, and
$\vol(W_{\ell},\beta_{\varepsilon})=
2a_{\ell}b_{\ell}$. We consider the annulus
\[ A_{\ell}:=\{*\}\times (\R/\Z)_{x_2}\times[-1,1]\subset T^2\times[-1,1]\cong
W_{\ell}\]
with area form $b_{\ell}\, \rmd x_2\wedge\rmd r$. Given $\delta>0$
(and smaller than~$1$), by a simple application of
Moser's method~\cite{mose65} we can choose an area-preserving
embedding
\[ \psi_{\ell}\co (r_{\ell}\D,\rmd x\wedge\rmd y)\longrightarrow
(A_{\ell},b_{\ell}\, \rmd x_2\wedge\rmd r)\]
with
\[ \pi r_{\ell}^2=(1-\delta)\area(A_{\ell}).\]
This extends to an embedding
\[ \begin{array}{rccc}
\Psi_{\ell}\co & r_{\ell}\D\times\R/\Z & \longrightarrow &
    A_{\ell}\times(\R/\Z)_{x_1}=W_{\ell}\\
               & \bigl((x,y),s\bigr)   & \longmapsto     &
    \bigl(\psi_{\ell}(x,y),s\bigr)
\end{array}\]
satisfying
\[ \Psi_{\ell}^*\bigl(a_{\ell}\,\rmd x_1-(b_{\ell}r+c_{\ell})\,\rmd x_2\bigr)
=a_{\ell}\,\rmd s+\lambda_{\ell}\]
with $\lambda_{\ell}:=\psi_{\ell}^*(-(b_{\ell}r+c_{\ell})\,\rmd x_2)$
a primitive of $\rmd x\wedge\rmd y$.

Let $\beta_{\ell}=\beta_{r_{\ell},\lambda_{\ell},\delta}$
be the contact form on $r_{\ell}\D\times\R/\Z$
provided by Proposition~\ref{prop:abhs}. Notice that because of
the factor $a_{\ell}$ in front of $\rmd s$, the minimal period
of the Reeb vector field will be bounded below by~$a_{\ell}$ rather
than~$1$.

Now we replace the contact form $\beta_{\varepsilon}$ on the
image of $\Psi_{\ell}$ by $(\Psi_{\ell})_*\beta_{\ell}$.
The image under $\Psi_{\ell}$ of the Morse--Bott foliation
on $r_{\ell}\D\times\R/\Z$ constructed in the preceding section,
near the image of the boundary $\partial(r_{\ell}\D)\times\R/\Z$, is the
product of a Morse foliation by circles in $A_{\ell}$ with
the $(\R/\Z)$-factor corresponding to the $x_1$-coordinate.
So the Morse foliation on $A_{\ell}$ shown in
Figure~\ref{figure:morse-foliationA}, with one additional critical point
of index~$1$, defines an extension
to a Morse--Bott foliation on~$W_{\ell}$ for the new contact
form $\beta_{\varepsilon,\delta}$ on~$M$.
The Morse function can be chosen in such a way that the new Morse--Bott
function on $W_{\ell}$ coincides with the original $f|_{W_{\ell}}=f(r)$
near $\{r=\pm 1\}$.

\begin{figure}[h]
\labellist
\small\hair 2pt
\pinlabel $\psi_{\ell}(r_{\ell}\D)$ at 250 205
\pinlabel $r=-1$ [l] at 509 3
\pinlabel $r=1$ [l] at 509 362
\endlabellist
\centering
\includegraphics[scale=0.45]{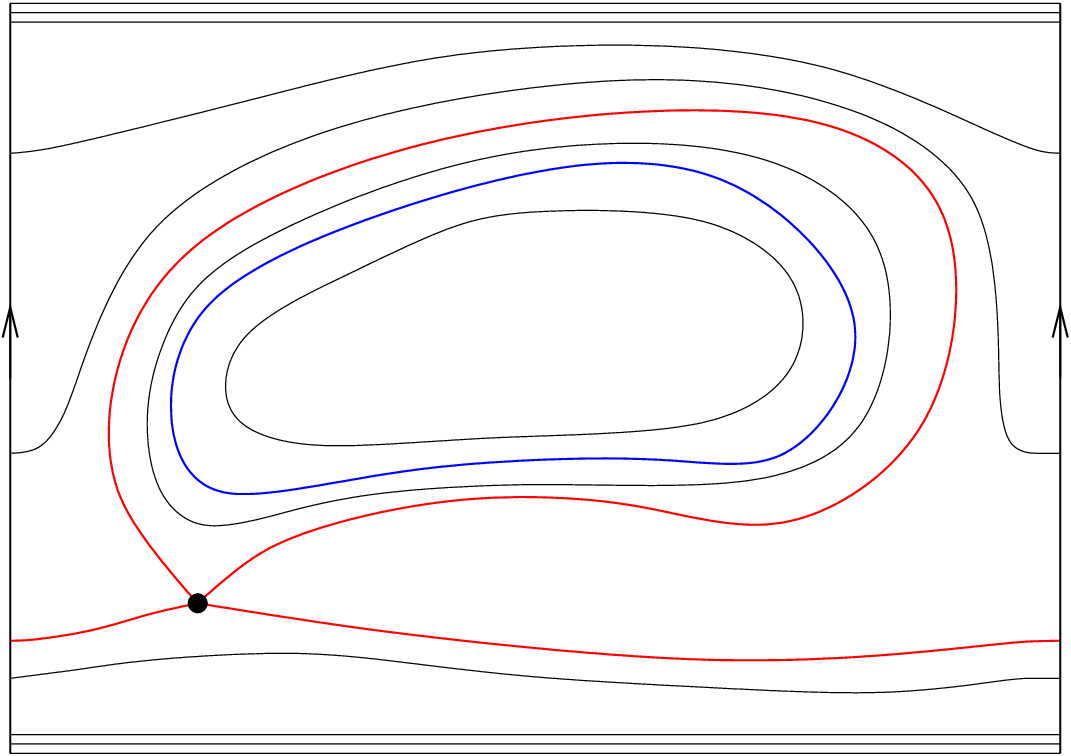}
  \caption{The Morse foliation on the annulus
           $A_{\ell}=\R/\Z\times [-1,1]$.}
  \label{figure:morse-foliationA}
\end{figure}

By construction, we have
\[ \Tmin(W_{\ell},\beta_{\varepsilon,\delta})\geq a_{\ell}
=\Tmin(W_{\ell},\beta_{\varepsilon}),\]
and hence
\[ \Tmin(M,\beta_{\varepsilon,\delta})\geq\Tmin(M,\beta_{\varepsilon})>
\Tmin(M,\alpha)-\varepsilon,\]
with $\Tmin(M,\alpha)-\varepsilon$ replaced by
the lower bound given at the end of Section~\ref{section:linear-Lutz}
in case $(M,\alpha)$ contains ILLs.

For the volume, we estimate
\begin{eqnarray*}
\vol (M,\beta_{\varepsilon,\delta})
 & = & \vol (U,\beta_{\varepsilon})+\sum_{\ell}\vol(W_{\ell},
       \beta_{\varepsilon,\delta})\\
 & < & \varepsilon+\sum_{\ell}(2a_{\ell}b_{\ell}+1)\delta.
\end{eqnarray*} 
Here the summand $\delta$ corresponds to the volume of
the image of $\Psi_{\ell}$ after the modification of
$\beta_{\varepsilon}$ into $\beta_{\varepsilon,\delta}$
and comes from part (iv) of Proposition~\ref{prop:abhs};
the summand $2a_{\ell}b_{\ell}\delta$ is the volume of the complement
of this image, where $\beta_{\varepsilon}$ is left unmodified.

The number of summands (i.e., the range of~$\ell$) and
the $a_{\ell},b_{\ell}$ depend on the chosen $\varepsilon>0$.
By then choosing $\delta>0$ sufficiently small, we can ensure that
$\vol(M,\beta_{\varepsilon,\delta})<2\varepsilon$. With
$\alpha_{\varepsilon}:=\beta_{\varepsilon,\beta}$ we then have the estimate
\[ \frac{\Tmin(M,\alpha_{\varepsilon})^2}{\vol(M,\alpha_{\varepsilon})}
>\frac{(\Tmin(M,\alpha)-\varepsilon)^2}{2\varepsilon}\]
for $\varepsilon<\Tmin(M,\alpha)$. Since $\varepsilon>0$ can be chosen
as small as we like, Theorem~\ref{thm:main} is proved.
\begin{ack}
We thank the referee of an earlier version of this
paper for pointing out that ILLs necessitate an additional argument.
J.~H.\ was supported by the Radboud Excellence Initiative and gratefully
acknowledges support from the Institute for Advanced Study
and an anonymous donor.
M.~S.\ is supported by the Deutsche Forschungsgemeinschaft on the project
``Reeb flows with symmetries: dynamics and topology''
(SA 5534/1-1, Projektnummer 568330991).
\end{ack}

\begin{AI}
We have not used LLMs, neither during the research that led to this paper,
nor for the preparation of the manuscript. 
H.~G.\ is a member of the Association for Human Mathematics
and its AI-free caucus.
\end{AI}

\end{document}